\documentclass[12pt]{amsart}

\usepackage{amssymb}
\usepackage{amsmath}
\usepackage{amsrefs}
\usepackage[all]{xy}
\usepackage{mathrsfs}
\usepackage{a4wide}
\usepackage{xcolor}

 \theoremstyle{plain}
\newtheorem{theorem}{Theorem}[section]
\newtheorem{lemma}[theorem]{Lemma}
\newtheorem{proposition}[theorem]{Proposition}

\newtheorem{definition}[theorem]{Definition}

\theoremstyle{definition}

\newtheorem{claim}{Claim}

\theoremstyle{remark}

\numberwithin{equation}{section}

\DeclareMathOperator{\dom}{dom}

\DeclareMathOperator{\ran}{range}

  \def\Poset{{\mathbb P}_2}

\newcommand{\card}[1]{\lvert #1 \rvert}

\title{Autohomeomorphisms of pre-images of $\mathbb N^*$}
\author{Alan Dow}

\address{Department of Mathematics and Statistics, UNC Charlotte}
\begin{document}

\begin{abstract} In the study of the Stone-{\u C}ech remainder of
the real line a detailed study of the Stone-{\u C}ech remainder 
of the space $\mathbb N\times [0,1]$, which we denote as
$\mathbb M$, has often been utilized.  Of course the real line can be covered by
two closed sets that are each homeomorphic to $\mathbb M$. 
It is known that an autohomeomorphism of $\mathbb M^*$
induces an autohomeomorphism of $\mathbb N^*$. 
 We prove that it is consistent with there being
 non-trivial autohomeomorphism of $\mathbb N^*$ 
 that those induced by autohomeomorphisms of
  $\mathbb M^*$ are trivial.
\end{abstract}

\maketitle

\section{Introduction}
We consider the 
remainder $\mathbb M^*
 = \beta(\mathbb N\times [0,1])\setminus
  (\mathbb N\times [0,1])
  =\beta\mathbb M\setminus \mathbb M$ 
 and the projection
 \\ 
   $\pi :\mathbb M^*\rightarrow
    \mathbb N^*$
    satisfying
    that $\pi((a\times [0,1])^*)
     = a^*$ for all infinite 
      $a\subset\mathbb N$.

For each $u\in \mathbb N^*$,
       we let $\mathbb I_u$ denote the pre-image $\pi^{-1}(u)$
       of $u$. Each $\mathbb I_u$ is a continuum, and it
       is an example of a standard subcontinuum as per \cite{Hart92}.
       Perhaps it is appropriate to refer to each of
       these (more special) $\mathbb I_u$ as primary standard
       subcontinuum.    
\bigskip

It was shown in \cite{DH1993} that every autohomeomorphism
 $\Psi$ of $\mathbb M^*$ induces an autohomeomorphism
  $H_\Psi$ of $\mathbb N^*$ satisfying that $\Psi(\mathbb I_u)
   = \mathbb I_{H_\Psi(u)}$ (i.e. $H_\Psi\circ \pi = \pi\circ \Psi$).
The question addressed in this article,  raised
in \cite{DH1993}, is  whether every autohomeomorphism
 $H$ of $\mathbb N^*$ is induced by some
 autohomeomorphism of $\mathbb M^*$ in this manner.
 It is evident that this question has an affirmative answer
 if all autohomemorphisms of $\mathbb N^*$ are trivial.
 We prove in Theorem \ref{P2} below that it is consistent
 that a negative answer is also consistent.
 The affirmative answer under CH will appear in another paper.
The main reference  for results in this section is \cite{Hart92}. 

\bigskip

Within any compact space $X$
and a sequence $\{ A_n : n\in \mathbb N\}$
of subsets,  it is common to
let
   $u{-}\lim\{A_n\}_n$ denote 
  the usual set of $u$-limits,
    $\bigcap\{ \overline{\left(
    \bigcup_{n\in a}A_n\right)} : a\in u\}$.

\bigskip

In particular, if 
 $\{ [a_m, b_m ] : m\in \mathbb N\}$
 is a sequence of pairwise disjoint
 connected subintervals of
  $\mathbb N\times I$, 
   we let $[a_m,b_m]_u$ 
   denote $u{-}\lim\{[a_m,b_m]\}_m$.
 If there is some $U\in u$
 satisfying that the sequence  
   $ \{[a_m,b_m] : 
   m\in U\}$ is locally finite in 
   $\mathbb N\times I$,
   then $[a_m,b_m]_u$ is also an example of a
   standard subcontinuum of
   $\mathbb M^*$. 

\bigskip

Say that a sequence
 $\{ [a_m, b_m] : m\in \mathbb N\}$
 is a standard sequence   if it is a locally finite
 set of pairwise disjoint non-trivial connected
 intervals in $\mathbb N\times I$. 
 A standard sequence will be called rational
 if all the end-points are rational numbers.

\begin{definition} If
$\mathcal A = \{ [a_m, b_m] : m\in \mathbb N\}$
 is a standard sequence,
 let
 $\mathbb M^*_{\mathcal A}$ denote the set
  $ \mbox{cl}_{\beta\mathbb M}(\bigcup\mathcal A)\setminus \mathbb M$.
 \end{definition} 

\begin{lemma} If $\mathcal A$ is a standard sequence,
then $\bigcup \mathcal A$ is homeomorphic to $\mathbb M$,
and 
$\mathbb M^*_{\mathcal A}$ is homeomorphic to
 $\mathbb M^*$.
\end{lemma}

\begin{proposition}  If $K$ and $L$ are disjoint\label{basis} compact subsets of $\mathbb M^*$,
then there are standard sequences 
$\mathcal A = \{ [a_m, b_m ] : m\in \mathbb N\}$ and
 $\mathcal C = \{ [c_m, d_m] : m\in \mathbb N\}$ 
 and disjoint sets $N_K, N_L$ of $\mathbb N$ such
 that 
 \begin{enumerate}
 \item for all $m\in\mathbb N$, $a_m< c_m < d_m < b_m$,
 \item $K$ is contained in the $\beta\mathbb M$-closure of $\bigcup \{ [c_m, d_m] : m\in N_K\}$
 \item $L$ is contained in the $\beta\mathbb M$-closure of $\bigcup\{[c_m, d_m] : m\in N_L\}$.
 \end{enumerate}
\end{proposition}

A point $x$ of $[a_m,b_m]_u$
is a cut-point 
if for every $f\in \mathbb N^{\mathbb N}$,
there is a
sequence
  $a_m < c_m < d_m < b_m$ ($m\in \mathbb N$)
   such that
   \begin{enumerate}
   \item $x$ is a cut-point of $[c_m,d_m]_u$,
   \item $\{ m\in 
    \mathbb N : d_m-c_m < 1/f(m)\}\in u$.
   \end{enumerate}

Say that a standard sequence $\{[c_m, d_m] : m\in \mathbb N\}$
is $f$-thin if $d_m-c_m < 1 / f(m)$ for all $m\in\mathbb N$.

 Certainly, if $a_m < x_m < b_m$
  (for all $m\in\mathbb N$),
   then $x_u = u{-}\lim\{x_m : m\in \mathbb N\}$
   is a (standard) cut-point of 
    $[a_m, b_m]_u$.  
   \bigskip

The standard cut-points of
 $[a_m, b_m]_u$ are 
 naturally linearly ordered 
 as in the ultraproduct. The
 closure of any subinterval 
 of the standard cut-points 
 is said to be a subinterval
 of $[a_m, b_m]_u$. 

   \bigskip

    If  
     $\{ [a_m, b_m]: m\in \mathbb N\}$
     is a standard sequence
     then for every selection
      $a_m\leq x_m \leq b_m$,
       the set
        $\{x_u : u\in \mathbb N^*\}$
    maps homeomorphically to $\mathbb N^*$
     by the map $x_u$ being sent to $u$.
     Say that a sequence $\{ x_m : m\in \mathbb N\}$
     is a selector sequence (for $\{ [a_m, b_m] : m\in \mathbb N\}$)
     if $a_m < x_m < b_m$ for all $m\in \mathbb N$.

     \bigskip

     \begin{proposition}
     If a subset $L$ of $\mathbb M^*$
     is homeomorphic to an interval
     of any standard subcontinuum
     of any standard sequence  
      then it is an actual subinterval 
      of a standard subinterval
      of a standard sequence.
     \end{proposition}

I believe this next result is new.
For any continuous $g:\mathbb M\rightarrow I$
let $g^*$ denote the natural extension of $g$
mapping $\mathbb M^*$ to $I$.

\begin{theorem} Suppose\label{new1} that
 $\Psi : \mathbb M^*\rightarrow \mathbb M^*$
 is a homeomorphism
 and let $\{ [a_m, b_m] : m\in \mathbb N\}$
 be a standard sequence in $\mathbb M$.
 Then, for any selector sequence
  $\{ x_m : m\in \mathbb N\}$
  for $\{ [a_m,b_m]: m\in \mathbb N\}$,
  there is a standard sequence
   $\{ [c_j, d_j] : j\in \mathbb N\}$
   and a homeomorphism $h: \mathbb N^*
   \rightarrow \mathbb N^*$ such that
   for each $u\in \mathbb N^*$,
   \begin{enumerate}
\item 
 $\Psi(x_u) \in [c_j,d_j]_{h(u)}$,
 \item $[c_j,d_j]_{h(u)}$ is a subinterval
 of $\Psi([a_m,b_m]_u)$,
\item the mapping $\Psi(x_u) $ to
 $h(u)$ is a homeomorphism from\\
  $\Psi(\{ x_m: m\in\mathbb N\}^*)
   = \{ \Psi(x_u) : u\in \mathbb N^*\}$
   to $\mathbb N^*$.
   \end{enumerate}  
\end{theorem}

 \begin{proof}
 Choose a continuous 
  function $g: \mathbb M\rightarrow I$
  satisfying that, for every $m\in \mathbb N$,
   $g(x_m)=1$ and $g(r)=0$ for all $r\notin
    \bigcup\{ (a_m,b_m): m\in \mathbb N\}$. 
\medskip

    Choose a continuous function $g_1 : \mathbb M\rightarrow I$
    satisfying that $g_1^* = g^*\circ \Psi^{-1}$. 
    For each $u\in \mathbb N^*$, let $y_u = \Psi(x_u)$. 
    Observe that $g_1^*(y_u)=1$ for all $u\in \mathbb N^*$. 
\medskip

Let $\{ [c_j , d_j ] : j\in \mathbb N\}$ 
enumerate all maximal connected intervals in $\mathbb M$
that satisfy that $(c_j,d_j)\cap g_1^{-1}(\frac12, 
1]$ is dense in $[c_j,d_j]$. 
\medskip

Consider any $u\in \mathbb N^*$ and the subcontinuum
 $\tilde I_u = \Psi(I_u)$. There is a 
 standard interval subsequence
   $\{ [\tilde a_m,\tilde b_m] : m\in \mathbb N\}$
   of $\{ [a_m, b_m] : m\in \mathbb N\}$ 
   that has   $\{ x_m : m\in \mathbb N\}$ as a selector
and is contained in $g^{-1}(\frac 34, 
1])$. The point $y_u$ is in $\Psi([\tilde a_m,\tilde b_m]_u)$
which is a subinterval of $\tilde I_u$. 
Also $\Psi([\tilde a_m,\tilde b_m]_u)$ 
is contained in the interior of $g_1^{-1}(\frac 12, 
1]$. 
Now choose any standard sequence $\{ [r_n, s_n]: n\in \mathbb N\}$
so that $g_1
\left(\bigcup\{[r_n,s_n]:n\in \mathbb N\}\right)^*\subset 
(\frac12, 
1]$
and so that $\Psi([\tilde a_m,\tilde b_m]_u)$ is a subinterval
of $[r_n,s_n]_v$ for some (unique) ultrafilter $v\in\mathbb N^*$.
\bigskip

For each $n\in \mathbb N$, choose the unique
 $j_n\in\mathbb N$ so that $[r_n,s_n]\subset (c_{j_n},d_{j_n})$.
 Since $\{ [r_n , s_n] : n\in\mathbb N\}$ is locally finite,
  the sequence $\{ [c_{j_n}, d_{j_n}] : n\in \mathbb N\}$
  is also locally finite. Hence
   $\{ [c_{j_n} , d_{j_n} ] : n\in \mathbb N\}$ is a standard
   sequence. The mapping $n\mapsto j_n$ is finite-to-one
   and now let $w$ be the finite-to-one image of
    $v$. Consider the standard subcontinuum
     $  [c_{j },d_{j } ]_w $.  We check that
      $ [c_{j },d_{j } ]_w $ is an interval
      in $\tilde I_u$, and that
       $h(u) =  w$   is the map 
    that satisfies the statement of the Lemma.
    \medskip

    The continuum  $[c_{j },d_{j } ]_w$
    is contained in the component
     $\tilde I_u$ of $\Psi\left(~\left(
     \bigcup\{[a_m,b_m]: m\in \mathbb N\}\right)^*\right)$
     and contains  the interval $\Psi([\tilde a_m,\tilde b_m]_u)$.
     It follows from the results in [Hart92], 
     that $[r_n,s_n]_v$ is an interval in
     $[c_j,d_j]_w$.  Choose any continuous
      function $g_2 : \mathbb M\rightarrow I$
      satisfying that $g_2^*(y_u)=1$ 
      and $g_2^*(\tilde I_u\setminus [r_n,s_n]_v)=0$.
      \medskip

      Set $L = \{ n : g_2([a_n,b_n]) \setminus [0,\frac14]\neq\emptyset\}$
      and $\tilde L = \{ j_n : n\in L\}$. 
      Since $g^*_2(y_u)=1$, it follows that $L\in v$ and $\tilde L\in w$.
  Let the function sending each $n$ to $j_n$ be denoted by $\rho$.
  Thus $\rho^*(v)=w$. For every $\tilde v\in \mathbb N^*$ such
  that $\rho^*(\tilde v) = w$, we have that
     $[r_n,s_n]_{\tilde v}$ is a subset of $[c_j,d_j]_w\subset
      \tilde I_u$. Therefore $v$ is the unique ultrafilter
      satisfying that $\rho^*(v)=w$ and $L\in v$.
      Therefore, by removing a finite set from $L$,
       we can assume that $\rho$ is 1-to-1 on  $L$. 
       \medskip

       Next, we note that $y_u$ is in the interior
       of $\left( \bigcup\{ [r_n,s_n] : n\in L\right)^*$
       (easily checked by the results in [Hart92])
       and so there is a $U\in u$ such 
       that $\Psi(\{x_m : m\in U\}^*)\subset
        \left( \bigcup\{ [r_n,s_n] : n\in L\right)^*$.
        
      \medskip

      We now have our desired mapping defined on $U^*$.
      For each $\tilde u\in U^*$, there is a unique
        $v_{\tilde u}\in L^*$ so that $y_{\tilde u}
         \in [r_n,s_n]_{v_{\tilde u}}$.
The continuous
projection mapping on $ \left( \bigcup\{ [r_n,s_n] : n\in L\right)^*$
to $L^*$ agrees with the mapping sending $y_{\tilde u}$ to $v_{\tilde u}$.
The mapping sending $\tilde u$ to $y_{\tilde u}$ is induced
by the continuous mapping $\Psi\restriction \{ x_m : m\in U\}^*$.
Therefore the mapping         
         $h(\tilde u)$ sent to   $\rho^*(v_{\tilde u})$
         is continuous (and 1-to-1).
         
 \end{proof}

    \section{$\mathbb N^*$ cut-sets}

\begin{definition}
Say that a subset $K$ of $\mathbb M^*$ is an $\mathbb N^*$
cut-set (for the standard
 sequence)
 if there is a standard sequence 
 $\mathcal A = \{ [a_m,b_m] : m\in \mathbb N\}$
 such that
 \begin{enumerate}
 \item $K$ is a subset $\mathbb M^*_{\mathcal A}$,
 \item $K\cap [a_m,b_m]_u$ is a cut-point of $[a_m,b_m]_u$
 for every $u\in \mathbb N^*$,
\item the  mapping from 
$\left(\bigcup\{[a_m,b_m]:m\in \mathbb N\}\right)^*$
where $[a_m,b_m]$ is sent to $m$,
 to $\mathbb N^*$ is 1-to-1 on $K$.
  \end{enumerate} 
An $\mathbb N^*$ cut-set $K$ 
is trivial if  $K$ equals $D^*$
for  some closed discrete $D\subset \mathbb M$.
\end{definition}

A concept called non-trivial \textit{maximal nice filters\/} 
on spaces of the form $\mathbb N\times X$ (with $X$ compact)
was
introduced and studied in \cite{DowPAMS} where it was shown
that PFA implies these do not exist if $X$ is metrizable. The neighborhood
filters in $\beta \mathbb M$ of a
 $\mathbb N^*$ cut-set result in maximal nice filters.

Let us a recall that a set $K$ is said to be a $P_\kappa$-set in a space $X$
if every $G_{<\kappa}$ of $X$ that contains $K$ has $K$ in its interior. It is
well-known that $D^*$ is a $P_{\mathfrak b}$-set for every closed subset
$D$ of $\mathbb M$. In this next result we prove that an 
 $\mathbb N^*$ cut-set has a neighborhood base resembling 
 that of a cut-point.  This result isn't strictly needed since the 
  $\mathbb N^*$ cut-sets that we intend to consider are 
  those of the form $\Psi(D^*)$ when $D^*$ is a trivial 
  $\mathbb N^*$ cut-set.

\begin{proposition}  Every $\mathbb N^*$ cut-set\label{Psets} $K$ of $\mathbb M^*$ is a
 $P_{\mathfrak b}$-set in $\mathbb M^*$. Also, if
  the standard sequence $\mathcal A = \{ [a_m, b_m]  : m\in \mathbb N\}$
witnessing that $K$ is an $\mathbb N^*$ cut-set, then there is a
family of rational standard sequences $\{ \mathcal C_f : f\in \mathbb N^{\mathbb N}\}$ 
such that, for each $f\in \mathbb N^{\mathbb N}$,

$\mathcal C_f = \{ [c^f_m,d^f_m] \subset [a_m,b_m] : m\in \mathbb N\}$
  is $f$-thin

\noindent and the family $\{ \left(\bigcup\mathcal C_f \right)^* : f\in \mathbb N^{\mathbb N}\}$
is a neighborhood base for $K$ in $\mathbb M_{\mathcal A}^*$. Note also
that the family $\{ \left(\bigcup\mathcal C_f\right)^* : f \in \mathbb N^{\mathbb N}\}
$
is ${<}\mathfrak b$-directed.
\end{proposition}

\begin{proof}  Let $\mathcal A = \{ [a_m, b_m]  : m\in \mathbb N\}$ be
a  standard sequence 
witnessing that $K$ is an $\mathbb N^*$ cut-set. Let $U$ be an 
 open subset of $\beta \mathbb M_{\mathcal A}$ that contains $K$.  
 Select, by Lemma \ref{basis},
 a pair $\tilde {\mathcal A} =\{ [\tilde a_m , \tilde b_m] : m\in \mathbb N\}$ 
 and $\tilde {\mathcal C} = \{ [\tilde c_m , \tilde d_m ]  : m\in \mathbb N\}$
 so that there are disjoint subsets $N_K, N_{U}$ of $\mathbb N$ so that
 \begin{enumerate}
 \item for all $m\in \mathbb N$, $\tilde a_m < \tilde c_m < \tilde d_m < \tilde b_m$,
 \item for all $m\in \mathbb N$ there is a $k$ such that $[\tilde a_m, \tilde b_m]\subset
    [a_k,b_k]$,
    \item $K$ is contained in the $\beta\mathbb M$-closure of $\bigcup\{[\tilde c_m,\tilde d_m] : m\in 
       N_K\}$,
       \item $\mathbb M_{\mathcal A}^* \setminus U_n$ is contained 
        in  the $\beta\mathbb M$-closure of $\bigcup\{[\tilde c_m,\tilde d_m] : m\in 
       N_{U_n}\}$.
 \end{enumerate} 
 Let $\psi$ denote the function on $\mathbb M_{\mathcal A}$
 onto $\mathbb N$ obtained by sending every point of $[a_m,b_m]$ to $m$. 
 We may let $\psi^*$ then denote the continuous extension of $\psi$
 to $\mathbb M_{\mathcal A}^*$ onto $\mathbb N^*$.   
 
 For each $x\in K$ (a cut-point) with $\psi^*(x)=u \in \mathbb N^*$, there is 
 a set $L_x\subset N_K$ such that $\psi\restriction L_x$ is 1-to-1 
 and $x$ is in the $\beta\mathbb M$-closure of the union of
 the standard
 sequence $\tilde {\mathcal A}_{L_x} = \{[\tilde a_m , \tilde b_m ] : m\in L_x\}$. 
 The closure of the $\tilde{\mathcal A}_{L_x}$ meets $K$ in a clopen
 set and so, by possibly shrinking $L_x$ we can ensure that $x$ 
 is still in the closure and that this clopen subset of $K$ is equal
 to the intersection of $K$ with the closure of the union
 of the set $\{ [a_j , b_j ] : j\in \psi[L_x]\}$ (i.e. all the $\mathcal A$-intervals
 that are hit by $\tilde{\mathcal A}_{L_x}$). Since there is a finite cover
 of $K$ by such clopen sets (associated with the closures of
 these $\tilde{\mathcal A}_{L_x}$'s), this shows that there is a selector
 function $\sigma: \mathbb N \rightarrow \mathbb N$ so that
 \begin{enumerate}
 \item for each $m\in\mathbb N$, $[\tilde a_{\sigma(m)}, \tilde b_{\sigma(m)}]$
 is a subset of $[a_m,b_m]$,
 \item the closure of the union of the intervals in $\{
 [\tilde a_{\sigma(m)}, \tilde b_{\sigma(m)}]: m\in \mathbb N\}$ contains $K$
 in its interior,
 \item the closure of the union of the intervals in $\{
 [\tilde a_{\sigma(m)}, \tilde b_{\sigma(m)}]: m\in \mathbb N\}$ is contained in $U_n$.
 \end{enumerate}

\noindent
Choose any function
 $f\in \mathbb N^{\mathbb N}$ so that
 for all but finitely many $m$, each of the distances 
  $\{ |\tilde a_{\sigma(m)}-\tilde c_{\sigma(m)}|, 
    |\tilde d_{\sigma(m)}-\tilde c_{\sigma(m)}|, 
   \{ |\tilde d_{\sigma(m)}-\tilde b_{\sigma(m)}|\}$ are greater 
   than $5/f(m)$.  For every $x\in K$, applying the definition of
   cut-point, choose an $f$-thin standard sequence $\mathcal C_x
    = \{ [c^x_m , d^x_m] : m\in \mathbb N\}$ so that
     $x \in [c^x_m, d^x_m]_{u}$ where $u=\psi^*(x)$.  
     Again, we can choose $L_x\subset \mathbb N$ so
     that $L_x\in u$ and so that for all $y\in K$
     with $L_x\in \psi^*(y)$, we also have that
       $y$ is in the $\beta\mathbb M$ closure of the
       union of $\mathcal C_x$.  In fact $L_x$ is simply any set
       such that $L_x^* = \psi^*(K\cap \mbox{cl}_{\beta\mathbb M}
       \left(\bigcup\{ [c^x_m, d^x_m] : m\in \mathbb N\}\right))$.
       If needed, we can make ${<}1/f(n)$ changes to any of the $c^x_m$'s and
  $d^x_m$'s so as to ensure every $y\in K$ with $L_x\in \psi^*(y)$, 
  is in the interior of 
  $\mbox{cl}_{\beta\mathbb M}
       \left(\bigcup\{ [c^x_m, d^x_m] : m\in \mathbb N\}\right)$.
Now consider the sequence 
  $\{
 [\tilde a_{\sigma(m)}, \tilde b_{\sigma(m)}]: m\in \mathbb N\}$.
  It follows that, for all but finitely many $m\in L_x$,
  the interval $[c^x_m,d^x_m]$ meets the interval $[\tilde c_{\sigma(m)},
  \tilde d_{\sigma(m)}]$. By the choice of $f$, we then have that, 
  for all but finitely many $m\in L_x$, the interval $[c^x_m,d^x_m]$
  is contained in $[\tilde a_{\sigma(m)}, \tilde b_{\sigma(m)}]$. 
  This shows that  
  $ 
       \left(\bigcup\{ [c^x_m, d^x_m] : m\in \mathbb N\}\right)^*$
       is contained in $U$ and 
       is a neighborhood of a relatively clopen set of points of $K$.
       By the compactness of $K$ there is an $f$-thin standard
       sequence $\mathcal C_f = \{ [c^f_m, d^f_m] \subset
        [a_m,b_m] : m\in \mathbb N\}$
       such that $\left(\bigcup\mathcal C_f\right)^*$ is contained in $U$. 
       
        Now we prove that the family
         is ${<}\mathfrak b$-directed. Now
         let $F\subset \mathbb N^{\mathbb N}$ be any 
         set with $|F|<\mathfrak b$.   For each $f\in F$,
         choose also $g_f\in \mathbb N^{\mathbb N}$ so
         that $f<g_f$ and for all $m\in \mathbb N$,
           $[c^m_{g_f},d^m_{g_f}]\subset  (c^m_f,d^m_f)$ (
           using that the
           closure of $D = \{ c^m_f, d^m_f : m\in \mathbb N\}$ 
           in $\beta\mathbb M$ is disjoint from $K$. 
           For each $f\in F$, let $h_f \in \mathbb N^{\mathbb N}$
            satisfy that, for all $m\in \mathbb N$, 
              $5/h_f(m)$ is less than each of the values
               $ c^m_{g_f}-c^m_f$, $d^m_f-d^m_{g_f}$, and 
                $ d^m_{g_f}-c^m_{g_f}$. 
           Choose 
           $\bar f\in \mathbb N^{\mathbb N}$ so that for all $f\in F$,
          $h_f <^* \bar f$.  Fix any $f\in F$, and note that
          there is some $m_f$ so that, for all $m>m_f$,
           we have that $[c^m_{\bar f},d^m_{\bar f}]\cap
             [c^m_{g_f},d^m_{g_f}]$ is not empty
             and that $h_f(m) < f(m)$.
             Consider any $m>m_f$ and the required
             condition: $c^m_{g_f}\leq d^m_{\bar f}$
             or $c^m_{\bar f}\leq d^m_{g_f}$. 
             In the first case we then have
              $$c^m_f< c^m_{g_f}-\frac1{\bar f(m)} < d^m_{\bar f}
               - \frac1{\bar f(m)} < c^m_{\bar f}< d^m_{\bar f}  < 
                c^m_{\bar f} + \frac1{\bar f(m)} 
                  < c^m_{g_f}+\frac1{g_f(m)} <d^m_{g_f}< d^m_{f}$$
                  which implies that $[c^m_{\bar f},d^m_{\bar f}]
                  \subset (c^m_f,d^m_f)$. 
                  The case when $c^m_f \leq d^m_{g_f}$ also
                  implies that  $[c^m_{\bar f},d^m_{\bar f}]
                  \subset (c^m_f,d^m_f)$ by a similar argument.         
\end{proof}

\begin{definition} Say that an\label{gap} indexed family
 $\{\mathcal C_f : f\in\mathbb N^{\mathbb N}\}$ is an
  $\mathbb N^*$ cut-set of $\mathbb M^*$ 
  \begin{enumerate}
  \item every $\mathcal C_f$ is a rational standard sequence indexed as
   $\{ [c^m_f, d^m_f] : m\in \mathbb N\}$,
   \item the sequence $\mathcal C_f$ is $f$-thin,
   \item the family is countably directed in the sense that, 
   for each family $\{ f_n : n\in \omega\}\subset \mathbb N^{\mathbb N}$,
    there is an $f\in \mathbb N$ such that, for all $n\in\omega$,
     $f<^*f_n$ and $\{ m : [c^m_f,d^m_f]\not\subset (c^m_{f_n},d^m_{f_n})\}$
     is finite.
  \end{enumerate}

Given any $\mathcal C_f$ and subset $I$ of $\mathbb N$,
 let $\mathcal C_f\restriction I =\{[c^m_f,d^m_f] : m\in I\}$. 
For an $\mathbb N^*$ cut-set family $
\mathfrak C = \{\mathcal C_f : f\in \mathbb N^{\mathbb N}\}$,
and a subset $I\subset\mathbb N$, 
 let $\mathfrak C\restriction I$ denote the 
 (re-indexed) $\mathbb N^*$ cut-set $\{ \mathcal C_f\restriction I : f
 \in\mathbb N^{\mathbb N}\}$. 
 Let $\mbox{Triv}(\mathfrak C)$ denote the ideal of subsets
 $I$ of $\mathbb N$ satisfying that the 
family $\mathfrak C\restriction I$ is trivial.
\end{definition}

\section{The effects of two consequences of PFA}

In this section we work with two assumptions
$(\dagger_1^+)$ and $(\dagger_2^+)$, each
of which is a consequence of PFA.
In fact, $(\dagger_2^+)$ is a consequence of
the open graph axiom (formerly OCA).
Let $H$ be an autohomeomorphism  of $\mathbb N^*$.
For a  subset $a$ of $\mathbb N$,
 $H$ is said to be trivial on $a^*$ if
 (by possibly removing a finite subset of $a$) 
 there is a bijection $h_a$ from $a$ into $\mathbb N$
 such that $H(c^*) = \left(h_a(c)\right)^*$
 for all $c\subset a$. 
  The family,
 $\mbox{Triv}(H) = \{ a\subset \mathbb N: 
 H \ \mbox{is trivial on }\ a^*\}$,
   is an ideal on $\mathcal P(\mathbb N)$.
 An ideal $\mathcal I$ on a countable set 
  $D$ is said to be ccc over fin, if
  given any uncountable family
  of pairwise almost disjoint subsets of $D$,
   all but countably many of them are
   in $\mathcal I$.

   Thus one says that an autohomeomorphism $H$ of
   $\mathbb N^*$  (or of $D^*$ for any countable 
   discrete set $D$)
   is trivial modulo ccc over fin,
   if the ideal $\mbox{Triv}(H)$ is ccc over fin.
   Similarly $H$ is somewhere trivial if
   there is some infinite set $I\in  \mbox{Triv}(H)$.

\begin{definition}
The statement $(\dagger_1^+)$ is the assertion
that every autohomeomorphism of $\mathbb N^*$
is trivial. The statement $(\dagger_1)$ is
the statement that every such autohomeomorphism
is trivial  modulo ccc over  fin. 
While we are at it, let $(\dagger_1^-)$ be 
the statement that every autohomeomorphism
of $\mathbb N^*$ is somewhere trivial.
 \medskip

The statement $(\dagger_2^+)$ is the assertion
that every $\mathbb N^*$ cut-set of $\mathbb M^*$
is trivial, and again $(\dagger_2)$ is 
the statement that every $\mathbb N^*$
cut-set of $\mathbb M^*$ is trivial
 on an ideal that is ccc over fin.
\end{definition}

\begin{lemma}  $(\dagger_2^+)$ holds\label{firstOGA}
if $\mathfrak c=\aleph_2$ and the
 open graph axiom holds. 
\end{lemma}

\begin{proof}
Let $\{\mathcal C_f : f\in \mathbb N^{\mathbb N}\}$ be  an
 $\mathbb N^*$ cut-set of $\mathbb M^*$.
 Using that the open graph axiom implies that $\mathfrak b=\aleph_2$
 and we are assuming $\mathfrak c = \aleph_2$, we
 can recursively 
  choose
  a dominating family $\{ f_\gamma : \gamma\in \omega_2\}
 \subset \mathbb N^{\mathbb N}$ 
 so that, by possibly making finite modifcations
 to each, the family 
  $\{ \mathcal C_{f_\gamma}  :\gamma\in \omega_2\}$
  is directed as in Proposition \ref{Psets} and
 is cofinal in the family $\{ \mathcal C_f : f\in \mathbb N^{\mathbb N}\}$.
  Let $F$ denote the family $\{f_\gamma : \gamma\in \omega_2\}$
  without the indexing.

   Let $R$ be the set of pairs $(f,f')\in [F]^2$
   such that there is an $m\in \mathbb N$ such that 
    $[c^m_f,d^m_f]\cap [c^m_{f'},d^m_{f'}]$ is empty.
     If we identify each $f\in F$ 
    with the corresponding sequence 
     $\mathcal C_f$ as an element of
      $\left(\mathbb Q^2\right)^{\mathbb N}$
      viewed as a product of discrete spaces,
      then the relation $R$ is open in the resulting metric space.
      Assume that $\{ f_\alpha : \alpha\in\omega_1\}\subset
      F$  satisfies
      that $(f_\alpha,f_\beta)\in R$ for all $\alpha\neq \beta\in \omega_1$.
      We may choose
      an $f\in F$ so that,
      for all $\alpha <\omega_1$, there is an $m_\alpha\in \mathbb N$
      so that $[c^m_f,d^m_f]\subset (c^m_{f_\alpha},d^m_{f_\alpha})$
      for all $m>m_\alpha$. 
       Fix an $\bar m\in\mathbb N$ and an uncountable
       $\Lambda\subset\omega_1$ so that $m_\alpha=\bar m$ 
       for all $\alpha\in\Lambda$. Also choose uncountable
         $\Lambda_1\subset \Lambda$ so
         that, for all $\alpha,\beta\in \Lambda_1$
         and $m\leq \bar m$, 
          $[c^m_{f_\alpha},d^m_{f_\alpha}]
          =
 [c^m_{f_\beta},d^m_{f_\beta}]$. We now have a contradiction since,
 for all $\alpha\neq\beta\in \Lambda_1$, 
 $[c^m_{f_\alpha},d^m_{f_\alpha}]\cap
[c^m_{f_\beta},d^m_{f_\beta}]$ is not empty (contradicting 
the pair $(f_\alpha,f_\beta)$ is supposed to be in $R$).

By the OGA, it follows there must be a cover of $F $
by a
countable collection 
 $\{ F_n : n\in \omega\}$
 satisfying that $[F_n]^2  $ is disjoint from $R$ for all
  $n\in\omega$. Choose any $n$ so that $F_n$ is a $<^*$-dominating
  subfamily of $\mathbb N^{\mathbb N}$. 
  By our construction, this also ensures that
   $\{ \mathcal C_f : f\in F_n\}$ is cofinal in 
    the original family $\{\mathcal C_f : f\in \mathbb N^{\mathbb N}\}$
    in the sense of Proposition \ref{Psets}.
  Fix any $m\in \mathbb N$
   and observe that the family $\{ [c^m_f,d^m_f] : f\in F_n\}$
   is linked. Of course this means we can choose  a sequence
    $\{ x_m : m\in \mathbb N\}$ satisfying
    that $x_m\in [c^m_f,d^m_f]$ for all $f\in F_n$
    and this contradicts that there should be some $f\in F_n$
    satisfying that the $\beta \mathbb M$-closure of $\{x_m\}_{m\in\mathbb N}$
    should be disjoint from
     the $\beta \mathbb M$-closure of $\bigcup\mathcal C_f$.
    
\end{proof}

\begin{theorem}[Assume $(\dagger_1)$ and $(\dagger_2)$]
Every autohomeomorphism\label{Histrivial}
 $\Psi$ on $\mathbb M^*$
induces a trivial autohomeomorphism on $\mathbb N^*$.
\end{theorem}

\begin{proof} 
Assume that $\Psi$ is an autohomeomorphism of $\mathbb M$. 
For simplicity assume that $\Psi(\,(\mathbb N\times \{0\})^*\,)
 = (\mathbb N\times \{0\})^*$.
 Let $H$ denote the autohomeomorphism on $\mathbb N^*$
 induced by $\Psi$, hence $H\circ \pi = \pi\circ \Psi$. 
 \medskip
 
Start with the continuous function $g: \mathbb M\rightarrow [0,1]$
satisfying that, for each $n\in \mathbb N$,  
\begin{enumerate}
\item 
 $g^{-1}(0)\cap (\{n\}\times [0,1]) = \{n\}\times \{ \frac in : 0\leq i\leq n\}$,
 \item
  $g^{-1}(1)\cap (\{n\}\times [0,1]) = \{ n\}\times \{ 
   \frac {2i+1}{2n} : 0\leq i < n\}$,
   \item  $g\restriction \left(\{n\}\times [\frac{i}{2n},\frac{i+1}{2n}]\right)$ is linear
   for all $0\leq i <2n$.
   \end{enumerate}
   Simply the graph of $g\restriction (\{n\}\times [0,1])$
   oscillates  $0$ to  $1$ and back to $0$ exactly $n$ times.
   Let $D=\{ (n,q) : n\in \mathbb N, q\in \{ \frac{i}{2n} : 0< i < 2n\}$.
   Of course $D$ is a closed discrete subset of $\mathbb M$.
   For an infinite subset $J$ of $\mathbb N$,
   let $D_J = D \cap \left(J\times [0,1]\right)$.
\medskip

    Let $g_1$ be a continuous map on $\mathbb M$ satisfying
    that $g_1^* = g^*\circ \Psi^{-1}$. For each $n\in \mathbb N$,
    let us count, call it $L_n$, the number of times that
     $g_1$ oscillates as follows. Fix any $n\in \mathbb N$
     and set $t^n_0 = 0$ and define (if it exists)
      $t^n_1 = \inf\{ r : g_1(n,r) = \frac23\}$
       (we think of $g_1$ as ``on its way towards 1''). 
       Next, we define $t^n_2$ to be 
       $\inf\{ r \in [t^n_1,1] : g_1(n,r) = \frac13\}$
       (i.e. ``$g_1$ is on its way back towards 0''). 
       Continue recursively defining this increasing
       sequence $t^n_0,t^n_1, \ldots, t^n_{L_n}$ 
       so that $g_1(n,t^n_i)=\frac13$ for even $i\leq L_n$
       and $g_1(n,t^n_{i})=\frac23$ for odd $i\leq L_n$. 
       Also, $g_1\restriction\left(\{n\}\times
        [t^n_i, t^n_{i+1}]\right)$ is
        either contained in $[0,\frac23]$ or
        in $[\frac13,1]$. 
For each $0\leq i < L_n$ also choose $s^n_i\in [t^n_i,t^n_{i+1}]$
so that, if $i$ is even, $g_1((n,s^n_i))=\min(g_1(\{n\}\times
[t^n_i,t^n_{i+1}])$, and, if $i$ is odd,
 $g_1((n,s^n_i))=\max(g_1(\{n\}\times
[t^n_i,t^n_{i+1}])$. 
It follows from the fact that $g_1^* = g^*\circ \Psi$,
that if $0 \leq i_n < L_n$ ($n\in \mathbb N$)
is a sequence with each $i_n$ even, 
then $\limsup\{ g_1((n,s^n_{i_n})) : n\in \mathbb N\} = 0$.
Similarly the limit would be $1$ for a sequence of odd
values for $i_n$. It also follows similarly,
 that the sequence $\{ L_n : n\in \mathbb N\}$
 is, mod finite, a diverging sequence of even numbers.
  For convenience (by possibly redefining finitely
  many) we assume $L_n$ is even for all $n\in \mathbb N$.

\bigskip

 Let $\sigma$ be the function in ${}^{\mathbb N}\,\mathbb N$
 defined by $\sigma(m)=L_m/2$ (it will help to
 change letters). 
Let us assume that $H$ is not trivial
 (meaning $\mathbb N$ is not in $\mbox{Triv}(H)$).
 We break the rest
 of the proof into cases based on properties
 of $\sigma$. 
 In each case we prove that there is
 an infinite $a\in \mbox{Triv}(H)$
 satisfying that $L_{h_a(n)}\neq n$
 for all $n\in a$. Appealing to symmetry
 we complete the proof in the case
 that $L_{h_a(n)}<n$ for all $n\in a$.
 Note that $(\dagger_1)$ implies that every
 infinite $b\subset \mathbb N$ has an infinite
 subset that is in $\mbox{Triv}(H)$.

 First case is when 
 $\sigma$ is, mod finite,
 a permutation on $\mathbb N$.  
 Since $\mathbb N\notin \mbox{Triv}(H)$,
  the almost permutation $\sigma^{-1}$ does
  not induce $H$. There is an infinite
   $b\subset \mathbb N$ such that 
   $\left(\sigma^{-1}(b)\right)^*$
   and $H(b^*)$ are distinct.  By
   taking complements if needed,
   we can assume that $\left(\sigma^{-1}(b)\right)^*$
   is not contained in $H(b^*)$.
   Then choose an infinite $J\subset b$
   so that $\left(\sigma^{-1}(J)\right)^*$
   is disjoint from $H(b^*)\supset
   H(J^*)$.  
Now by $(\dagger_1)$ we can assume that
 $J\in \mbox{Triv}(H)$ and choose
the injection $h_J : J \rightarrow \mathbb N$
witnessing that $a\in \mbox{Triv}(H)$.
It follows that, for every $n\in J$,
  $L_{h_J(n)} \neq n$.    

  \bigskip

  Next case is that there is some infinite 
   $b\subset \mathbb N$ such that
    $b$ is disjoint from $\{ L_m : m\in \mathbb N\}$.
    Again choose an infinite $J\subset b$
    in $\mbox{Triv}(H)$ and we again have
    $L_{h_J(n)}\neq n$ for all $n\in J$. 
     
    \bigskip

    The final case is that $\sigma$ is not 
     1-to-1.  Choose any disjoint pair $b_0, b_1$
     of infinite subset of $\mathbb N$
     satisfying that $\sigma$ is 1-to-1 on each
     while $\sigma[b_0] = \sigma[b_1]$. 
     First choose an infinite $a_0\subset b_0$
     with $a_0\in \mbox{Triv}(H)$. 
     Next choose $a_1\subset b_1$ so that
      $a_1\in\mbox{Triv}(H)$ and
       $\sigma[a_1]\subset \sigma[a_0]$. 
       Shrink $a_0$ so that $\sigma[a_0]=\sigma[a_1]$.
We may choose $J\in \mbox{Triv}(H)$ to be a subset of one of $a_0,a_1$ so that
 again $L_{h_a(n)}\neq n$ for all $n\in a$.

   \bigskip

Now we continue the proof by assuming that $J$
is an infinite set in $\mbox{Triv}(H)$ satisfying
that $L_{h_J(n)} < n$ for all $n\in J$. 
\medskip

   Choose a standard sequence (indexed by $D$) $\{
   [a_d , b_d ] : d\in D\}$ so that $d\in (a_d,b_d)$
   for all $d\in D$. Of course $D$ is a selector set
   for the standard sequence $\mathcal A 
    = \{ [a_d, b_d] : d\in D\}$ and
   $D^*$ is an $\mathbb N^*$
   cut-set for $\mathcal A$.  We pass to the subset
   $D_J = D\cap \left(J\times[0,1]\right)$
   and $\mathcal A_J = \{ [a_d, b_d] : d\in D_J\}$. 
   More generally, for $I\subset J$, 
   let $\mathcal A_I = \{ [a_d , b_d] : d\in D_I\}$. 
   
   By considering an uncountable almost disjoint family
   of infinite subsets of $J$ and using Theorem \ref{new1}
   we have that 
   $(\dagger_2)$ implies that
   the $\mathbb N^*$ cut-set $\Psi(D_J^*)$ 
   is trivial on an ideal that is ccc over fin,
    there is an infinite $I\subset J$ so that 
    $\Psi(D_I^*) = K$  
     is a trivial   $\mathbb N^*$ cut-set
    with respect to the homeomorphic copy of
     $\mathbb M^*$ we get from $\left( H(I)\times [0,1]\right)^*$. 
    Therefore we may choose a closed discrete $E\subset 
 H(I)\times [0,1]$  so that $E^* = K$. Notice
    that $\Psi\restriction D_I^*$
    is an homeomorphism from $D_I^*$ to $E^*$. 
Again, by using $(\dagger_1)$  and by possibly 
further shrinking $I$ in the same manner that we shrunk
$J$ to obtain $I$, we  can assume that  $\Psi
     \restriction\left(D_I^*\cap \mathbb M_I^*\right)$
     is trivial.  That means there is a
     (mod finite) bijection 
     $f: D_I \rightarrow E\cap \mathbb M_{H(I)}$
     inducing $\Psi$ on $D_I^*\cap \mathbb M_I^*$.  We omit the easy verification
     that, by removing a finite set from $I$, we have that
      for $d_1 < d_2 \in D\cap \left(\{n\}\times [0,1]\right)$ ($n\in I$), 
      $f(d_1)<f(d_2) \in E\cap \left(\{h_I(n)\}\times [0,1]\right)$.
    
 \bigskip

 Fix any $n\in I$, and for each $0\leq i \leq 2n$,
  let $e_{i,n} = f((n,\frac{i}{2n}))$ (noting
  that $(n,\frac{i}{2n})\in D_I$). 

\begin{claim} For all but finitely many $n\in b$,
 \begin{enumerate}
\item $g_1(e_{i,n}) < \frac13$ for all even $i\leq 2n$, and
\item $g_1(e_{i,n}) > \frac23$ for all odd $i<2n$.
 \end{enumerate}
\end{claim}

Indeed, $g^*$ will send every point of
  $\left( \{ (n,\frac{2i}{2n}) : n\in b, i\leq n\}\right)^*$ to
   $0$, and so $g_1^*$ must send every point
    of $\left( \{f
((n,\frac{2i}{2n})) = e_{2i,n} : n\in b, i\leq n\}\right)^*
   $ to $0$. The analogous property holds 
   for the set of   $e_{2i+1,n}$ ($n\in b$, 
    $i<n$). It thus follows that $g_1\restriction \left(
     h_a(n) \times [0,1]\right)$ must oscillate at 
     least $n$ times, contradicting that
     the oscillation number $L_{h_a(n)}$ is less than $ n$.
\end{proof}

\section{PFA}

In this section, in working towards our main result,
we give an alternate proof of
Vignati's (\cite{Vignati}) theorem that PFA implies that
every autohomeomorphism of $\mathbb M^*$ is trivial.
We will
 work with 
 $(\dagger_1^+)$ and $(\dagger_2^+)$ 
  and we  also assume the principle defined next.

 \begin{definition}
 Say that $\mathcal H$ is an $\omega^\omega$-family 
 if $\mathcal H = \{ h_f : f \in \omega^\omega\}$ is a family
 of functions satisfying simply that $\dom(h_f) =  \{ (n,m) \in \omega^2 
  :   m < f(n) \}$. We then say that such a family $\mathcal H$
  is coherent, if whenever $f <^* g$ are in $\omega^\omega$,
   then $\{ (n,m) \in \dom(h_f) \cap \dom(h_g) : h_f((n,m))\neq
    h_g((n,m)) \}$ is finite. 
    \medskip
 
Say that the principle $\omega^\omega$-cohere  holds if
each $\omega^\omega$-family $\mathcal H$ that is coherent,
 there is a function $h$ with domain $\omega\times\omega$
 such that $\mathcal H\cup \{h\}$ is also coherent.
    \end{definition}
   
   The principle $\omega^\omega$-cohere (not so named)
   is a well-known consequence of OCA due
   to Todorcevic (see \cite{Farah2000}*{2.2.7}).  It is well-known (essentially 
   due to Hausdorff)
   that the principle $\omega^\omega$-cohere implies that $\mathfrak b > \omega_1$. 
   Todorcevic also proved the following Proposition.
   
\begin{proposition} If $\omega_1<\mathfrak b$, then for any
coherent
$\omega^\omega$-family $\mathcal H$, there is a countable set that
mod finite contains the range of each $h\in \mathcal H$.
\end{proposition}   

\begin{proof} Assume there is no such countable set. Recursively
choose $\{ f_\alpha : \alpha < \omega_1\}\subset \omega^\omega$, 
 so that the range of $h_{f_\alpha}$ is not, mod finite, contained
 in the union of the ranges of the family $\{ h_{f_\beta} : \beta < \alpha\}$.
 Choose any $f\in \omega^\omega$ satisfying that 
    $f_\alpha <^* f$ for all $\alpha < \omega_1$. Let $L$ denote
    the range of $h_{f}$ and choose $\delta<\omega_1$ large enough
    so that any point of $L$ that is in the range of any $h_{f_\alpha}$
     ($\alpha < \omega_1$) is already in the range of some
     $h_{f_\alpha}$ with $\alpha < \delta$. We now have
     a contradiction to the coherence assumption that
     guarantees that $h_{f_\delta}$ is, mod finite,
     contained in $h_f$. 
\end{proof}
   
\begin{theorem} If\label{Trivial} $(\dagger_1^+),(\dagger_2^+)$ and the principle 
$\omega^\omega$-cohere hold, 
then every autohomeomorphism of $\mathbb M^*$
is trivial.
\end{theorem}

\begin{proof} Let $\Psi$ be an autohomeomorphism of $\mathbb M^*$. Since
we are assume $(\dagger_1^+)$, there is no loss to assume that 
   $\pi = \pi\circ \Psi$ (i.e. $\Psi$ is a lifting of the identity map on $\mathbb N^*$). 
   Fix an enumeration, $\{ q_\ell : \ell\in \omega\}$, of the rationals in $[0,1]$. 
   For every $f\in \omega^{\mathbb N}$, let $D_f 
    = \{ (n,q_\ell) : n\in \mathbb N, \ell < f(n)\}$.  Before proceeding
    we note that the collection $\bigcup \{ (D_f)^* : f\in \omega^{\mathbb N}\}$ is
     a dense subset of $\mathbb M^*$. By $(\dagger_2)$, we may fix, 
     for each $f\in \omega^{\mathbb N}$, 
     a countable
     set $E_f \subset \mathbb M$, satisfying that $(E_f)^* = \Psi((D_f)^*)$.
     Then by $(\dagger_1^+)$, we may fix a lifting $\sigma_f : D_f \rightarrow E_f$
     that induces the homeomorphism $\Psi\restriction (D_f)^*$.  
     \medskip
     
     Loosely identifying $\mathbb N$ and $\{ q_\ell : \ell < \omega\}$ with $\omega$,
      it is apparent that $\mathcal H  = \{ \sigma_f : f\in \omega^{\mathbb N}\}$ 
      can be regarded as an $\omega^\omega$-family.  Since $\sigma_f$ and 
       $\sigma_g$ (for $f<^*g$) induce the same mapping on $(D_f)^*$
       (i.e. the mapping $\Psi$), it follows that $\mathcal H$ is a coherent family. 
       By the principle $\omega^\omega$-cohere, we may choose a function 
        $\sigma: \mathbb N\times \{ q_\ell : \ell\in \omega\} \rightarrow \mathbb M$
        satisfying that $\sigma_f\subset^* \sigma$ for all $f\in \omega^{\mathbb N}$.         
        It follows that $\sigma$ is, mod finite, 1-to-1 and
        $\sigma\restriction \left(\{n\}\times I\right)$ is an
        order-preserving preserving function into $\{n\}\times I$
        because 
        any failure would violate $\sigma_f\subset^* \sigma$ for some suitably
        large $f$. Similarly, it is easily shown that the range of $\sigma$ is dense
        in $\{n\}\times I$ for all but finitely many $n\in\mathbb N$. 
        \medskip
        
        The final thing to show is that, for all but finitely many $n\in\mathbb N$,
         $\sigma\restriction \left(\{n\}\times [0,1]\right)$ is the restriction of
         a homeomorphism $g_n : \{n\}\times [0,1] \rightarrow \{n\}\times [0,1]$,
         and that the resulting (almost homeomorphism) $g = \bigcup_n g_n$ from
          $\mathbb M$ to $\mathbb M$ satisfies that the Stone-{\u C}ech extension
          $\beta g$ contains $\Psi$.  By continuity and density, it suffices to show
          that for any sequence $R =\{ (n,r_n) : n\in \mathbb N\}\subset \mathbb M$,
           for all but finitely many $n\in\mathbb N$, $g_n$ is continuous at $(n,r_n)$
           and that $\beta g
           \restriction  R^*  = \Psi\restriction R^*$.  
           This we do now.   Given such a set $R$ (an $\mathbb N^*$ cut-set), we
           can choose a sequence $S= \{ (n,s_n) : n\in \mathbb N\}\subset \mathbb M$
           satisfying that $\Psi(R^*) = S^*$.  Suppose there is an infinite
           set $a\subset \mathbb N$ satisfying that the oscillation of $\sigma$ on
           each neighborhood of $(n,r_n)$ is greater than some $\epsilon_n >0$. 
           Apply the continuity of $\Psi$ so as to choose a  standard sequence
            $\{ \{n\}\times [c_n,d_n] : n\in a\}$ satisfying that $c_n < r_n < d_n$ for all
             $n\in a$ and so that $\Psi$ sends the set $\left(\bigcup_{n\in a}
             \{n\}\times [c_n,d_n]\right)^*$
             into the set $\left( \bigcup_{n\in a}
             \{n\}\times (s_n-\epsilon_n/4,s_n+\epsilon_n/4)\right)^*$.
             Next choose two functions  $\rho_1, \rho_2\in \omega^{\mathbb N}$ satisfying
             that, for all $n\in a$, $\{q_{\rho_1(n)},q_{\rho_2(n)}\}\subset (c_n,d_n)$
             and so that $| \sigma((n,q_{\rho_1(n)})) - \sigma((n,q_{\rho_2(n)}))| > \epsilon_n$.
             Choose any $f\in \omega^{\mathbb N}$ large enough so that 
              $\rho_1(n)+\rho_2(n) < f(n)$ for all $n\in \mathbb N$. 
             Clearly there is an infinite set $b\subset a$ such that, by symmetry,
               $\sigma((n,q_{\rho_1(n)}))\notin [s_n-\epsilon_n/4,s_n+\epsilon_n/4]$ 
               for all $n\in b$.  But now the set $R_1 = \{ (n,q_{\rho_1(n)}) : n\in b\}$
               is a subset of $R_2 = D_f\cap \left(\bigcup_{n\in a} \{n\}\times [c_n,d_n]\right)$.
               Since $\Psi\restriction R_2^*$ is induced by $\sigma$, we have
                a contradiction since $(\sigma(R_1))^*$
                and $\Psi (R_2^*)$ are disjoint.            
\end{proof}

    \section{Non-trivial autohomeomorphisms}
 In this section we return to the question 
 raised in \cite{DH1993}:
``Given an autohomeomorphism
 $H:\mathbb N^*\rightarrow \mathbb N^*$,
 does there exist an autohomeomorphism
 $\Psi: \left(\mathbb N\times [0,1]\right)^*
 \rightarrow
  \left(\mathbb N\times [0,1]\right)^*$
  satisfying $H\circ \pi = \pi\circ \Psi$?'' 
  We prove that it is consistent to have 
   $(\dagger_1)$ and $(\dagger_2)$ holding in a model in which
   there is a non-trivial autohomeomorphism of $\mathbb N^*$
    (i.e. $(\dagger_1^+)$ fails). Once we succeed, then
    the following is a consequence of Theorem \ref{Histrivial}.
    
    \begin{theorem} There is a model in which there\label{P2}
    are non-trivial autohomeomorphisms of $\mathbb N^*$
    and every automorphism on $\mathbb N^*$ induced
    by an   autohomeomorphism   of
     $\mathbb M^*$  is trivial.
    \end{theorem}
    
    The model was introduced by Velickovic  (\cite{VelOCA}) and we
    will use the further analysis of the forcing initiated in 
    \cite{ShStvel}.
    We do not know
    if $(\dagger_2^+)$ holds in this model.  It would
    be interesting to have a   formulation of a suitable
    weakening of OCA that can be shown to hold in this model.

First we define
 the  partial order $\Poset$ from \cite{VelOCA}

\begin{definition}
    The\label{poset} 
partial order $\Poset$ is defined to consist of all 1-to-1
    functions 
$f$  where 
\begin{enumerate}
\item $\dom(f) = \ran(f)\subset \mathbb N$,
\item for all $i\in\dom(f)$ and $n\in\omega$, $f(i)\in
  [2^n,2^{n+1})$ if and only if 
 $i\in [2^n,2^{n+1})$
  \item $\limsup_{n\rightarrow\omega}\card{[2^n,2^{n+1})\setminus
\dom(f)} = \omega$\label{growth}
\item for all $i\in \dom(f)$, $i=f^2(i)\neq f(i)$.
\end{enumerate} 
The ordering on $\Poset$ is $\subseteq^{\ast}$.
\end{definition}

It is  shown in \cite{VelOCA}, see also \cite{ShStvel}, that
 $\Poset$ is $\sigma$-directed closed.
 The following partial order was introduced in \cite{ShStvel} as a great
tool to uncover the forcing preservation properties of $\Poset$, such
as the fact that 
 $\Poset$ is $\aleph_2$-distributive
(and so introduces no new $\omega_1$-sequences of subsets of $\mathbb
N$).

\begin{definition}[\cite{ShStvel}*{2.2}]
Given $\{ p_\xi :\xi\in\mu\} $, define 
   $
   \Poset(\{p_\xi : \xi\in\mu\})$  
to be the partial order consisting of all $q\in
\Poset$ 
such that there is some $\xi\in \mu$  such that
 $q =^* p_\xi$.
   The ordering on  $\Poset
   (\{p_\xi : \xi\in\mu\}) $ 
is $p\leq q$ if $p\supseteq q$ (rather than $p\supseteq^* q$ for $\Poset$).
\end{definition}

\begin{lemma}[\cite{step.29}]
In the forcing extension, $V[H]$,
 by $2^{<\omega_1}$,
  there\label{getF} is a maximal $\subset^*$-descending 
  sequence 
$\{ p_\xi : \xi\in\omega_1\}\subset\Poset $ which
  is $\Poset$-generic over $V$ and for which 
$\Poset(\{p_\xi :\xi\in\omega_1\})$ is ccc, $\omega^\omega$-bounding, 
and preserves that $\mathbb R\cap V$ is not meager.
\end{lemma}

\begin{proposition} If $\xi\in\omega_1$ and $\{p_\xi : \xi\in \mu\}$ is 
a descending sequence in $\Poset$, then $\Poset(\{p_\xi : \xi\in \mu\})$
is a countable atomless poset. 
\end{proposition}

\begin{lemma}[\cite{ShStvel}*{2.4}]
 Given $\eta\in\omega_1$, a $\subset^*$-descending sequence\label{oracle} 
 $\{p_\xi : \xi \in\eta\}\subset \Poset$, and a countable elementary
 submodel $\mathfrak U\prec (H(\aleph_2),\in)$,  such that
  $\{ p_\xi :\xi\in\eta\}\in \mathfrak U$, then there is a 
   $p\in \Poset$ which is $\mathfrak U$-generic for $\Poset
   (\{p_\xi : \xi\in \eta\})$. Moreover, for any extension
    $\{ p_\xi : \xi\in \mu\}\subset \Poset$ (again
     $\subset^*$-descending) such that $\eta<\mu$ and $p_\eta=p$,
     every $D\in \mathfrak U$ is predense in $\Poset(\{p_\xi :\xi\in\mu\})$
     provided it is dense in $\Poset(\{p_\xi : \xi\in \eta\})$.
\end{lemma}

Almost all of the work we have to do is to establish additional
preservation results for the poset $\Poset(\{p_\xi:\xi\in\omega_1\})$. 
The sequence $\{ p_\xi : \xi\in\omega_1\}$ (chosen in the
forcing extension by $2^{<\omega_1}$) will always be assumed
to be $\Poset$-generic over the PFA model. Following
 \cite{ShStvel}, let $\mathcal F$ be the filter on 
  $\Poset $ generated by the sequence
   $\{ p_\xi : \xi\in\omega_1\}$. Note that $V[\mathcal F]$
   is a generic extension of $V$ meaning that $\{ p_\xi : \xi\in
   \omega_1\}$ selects an element of every maximal antichain
   of $\Poset$ that is an element of $V$  but it is chosen within
   the model of CH and also introduces an $\omega_1$
   sequence cofinal in $\omega_2$. We may assume,
   as per \cite{ShStvel},
   that many statements about
   $\Poset$-names of Borel subsets of $\mathbb M$
    that are forced by $\Poset$ to hold will hold
    in $V[\mathcal F]$. In addition, there are no new
    Borel subsets of $\mathbb M$.

Once
these are established, we are able to apply the standard PFA type
methodology as demonstrated in \cite{ShStvel, step.29}. The technique
is to construct a $\Poset(\{ p_\xi : \xi\in\omega_1\})$-name
$\dot Q$  of
a proper poset and to then invoke PFA in the ground model
so as to select  a filter meeting a given choice of $\omega_1$-many 
dense
subsets of $2^{<\omega_1}*\Poset(\{p_\xi : \xi\in\omega_1\})*\dot Q$.
Using the proof that $\Poset$ is $\omega_1$-distributive, there is
then a condition $p\in \Poset$ that shows that simply
forcing with $\Poset$ yields
 the desired conclusion
from meeting those $\omega_1$-many dense sets.

For example, it is shown in \cite{step.29} that $\mathbb P$
forces  that $\mbox{Triv}(H)$ is a dense $P$-ideal (i.e.
$(\dagger_1^-)$ holds). 
However, since this is weaker than $(\dagger_1)$ we
refer to the following theorem to assert that $\mathbb P$
forces that $(\dagger_1)$ holds.  

\begin{theorem}[4.17 of \cite{Dowcccoverfin}] 
In the extension obtained by forcing over a model of PFA,
if $\Phi$ is a homomorphism from $\mathcal P(\mathbb N)/\mbox{fin}$
onto $\mathcal P(\mathbb N)/\mbox{fin}$, then $\mbox{Triv}(\Phi)$ is
a ccc over fin ideal.
\end{theorem}

Now we adapt the construction from \cite{ShStvel}
of
 $\{ p_\xi : \xi\in\omega_1\}$ so that we also have
 that $(\dagger_2)$ holds. The approach also incorporates
 ideas from Velickovic's proof that OCA and $\mathbb{MA}(\omega_1)$
 implies $(\dagger_1)$. 
 
 \begin{theorem} In the extension obtained by forcing
 over a model of PFA by $\Poset$, 
 every $\mathbb N^*$ cut-set of $\mathbb M^*$ is
 trivial on an ideal that is ccc over fin.
 \end{theorem}

 \begin{proof} Let $\{\dot {\mathcal C}_f : f\in \mathbb N^{\mathbb N}\}$
 be a family of $\Poset$-names that is forced (by $\mathbf{1}$) to be an
  $\mathbb N^*$
 cut-set of $\mathbb M^*$. 
 Much as in Theorem \ref{firstOGA}, we can assume that the family
  $\{\dot{\mathcal C}_f : f\in \mathbb N^{\mathbb N}\}$ is also
  forced to be order-preserving in the sense that
   $\left(\bigcup \dot{\mathcal C}_f\right)^*$ 
  contains $\left(  \bigcup\dot{\mathcal C}_g\right)^*$
   whenever $f <^* g$ are in $\mathbb N^{\mathbb N}$.
 
 Suppose also that the ideal of sets on which this cut-set is trivial is not
 ccc over fin. Since $\Poset$ is $\aleph_1$-distributive, we may choose
 a condition $p_0\in\Poset$ and an almost disjoint family
  $\{ I_\alpha : \alpha \in \omega_1\}$ of subsets of $\mathbb N$
  such that $p_0$ forces that none of the $I_\alpha$'s are in the
  trivial ideal for this cut-set. By  \cite{VelOCA} (since this final
  model is a model of Martin's Axiom) it suffices to assume
  that
  the family $\{ I_\alpha : \alpha <\omega_1\}$ is tree-like.
  More specifically,
   there is a function $\sigma :\mathbb N\rightarrow 2^{<\omega}$
   such there is a 1-to-1 enumerated
   collection $\{ \rho_\alpha: \alpha <\omega_1\}
   \subset 2^\omega$  so that
    for all $\alpha<\omega_1$,
      $\sigma(I_\alpha) \subset \{ \rho_\alpha\restriction j : j\in\omega\}$.  
  
  To start the proof, 
  let $p_0\in G$ be
  a $\Poset$-generic filter.  
    We consider the family $F$ of all partial functions
     $f\restriction I_\alpha$ ($f\in \mathbb N^{\mathbb N}$ and
      $\alpha\in\omega_1$). We define the relation $R\subset [F]^2$
      to consist of all unordered pairs $\{f,f'\}$  that satisfying
      \begin{enumerate} 
      \item if $\dom(f)=I_\alpha$ and $\dom(f')=I_\beta$, 
        then $\alpha\neq\beta$,
        \item there is an $m\in \dom(f)\cap \dom(f')$ 
        such that $[c^m_f,d^m_f]\cap [c^m_{f'},d^m_{f'}]$ is empty.
      \end{enumerate}
      We utilize the discrete topology on $S=\{\infty\}\cup
      \mathbb Q^2$ 
      and for any $f \in F$ we identify
       $f$ with an element, $s_f$, of the product space
        $S^{\mathbb N}$, 
         where in coordinate $m\in \dom(f)$,
          $s_f(m)=(c^m_f,d^m_f)$ and
          for $m\in \mathbb N\setminus \dom(f)$,
           $s_f(m)=\infty$. 
      With this topology, using tree-like as in \cite{VelOCA},
      the relation $R$ is an open subset
      of the product space $[\{s_f : f\in F\}]^2$ (i.e. unordered pairs).
      Assume that $\Lambda$ is an uncountable subset
      of $\omega_1$ and that, for each $\xi\in \Lambda$ 
      we set $\bar f_\xi = f_\xi\restriction I_\xi$. 
      The set $\{ \bar f_\xi : \xi\in \Lambda\}$ would
      be called an $R$-homogeneous set if it satisfied
      that       
      $\{ \bar f_\xi,\bar f_\eta\}\in R$ for all $\xi\neq\eta\in \Lambda$.
      We show
      that $\{ \bar f_\xi : \xi\in \Lambda\}$ fails to be an
      $R$-homogeneous set. 
      
      Choose any $f\in \mathbb N^{\mathbb N}$ so that
       $f_\xi <^* f$ for all $\xi\in \Lambda$. 
       Choose $\bar m\in \mathbb N$ and 
       uncountable $\Lambda_1\subset \Lambda$
       so that $[c^m_f,d^m_f]\subset [c^m_{f_\xi}, d^m_{f_\xi}]$
       for all $\bar m<m\in \mathbb N$. 
Choose also uncountable $\Lambda_2\subset\Lambda_1$ so that
 $\{ [c^m_{f_\xi},d^m_{f_\xi}] : m\leq \bar m\}
  =
  \{ [c^m_{f_\eta},d^m_{f_\eta}] : m\leq \bar m\}$  for all $\xi,\eta\in \Lambda_2$. 
  These reductions ensure that $\{\bar f_\xi,\bar f_\eta\}$ is not in 
  $R$ for any pair $\xi,\eta\in \Lambda_2$. 
  
  Suppose now we prove that there is a sequence $\{ p_\xi : \xi\in\omega_1\}$
  as in Lemma \ref{getF} that, in addition, ensures that
   $\Poset(\{p_\xi : \xi\in\omega_1\})$ forces that
   there is a proper poset $Q$ that \textbf{does} force there is an 
   uncountable $R$-homogeneous set. Then by our discussion in the paragraph
   following Lemma \ref{oracle} we would have contradicted  that
    $\{\dot {\mathcal C}_f : f\in \mathbb N^{\mathbb N}\}$
     is forced by $\Poset$ to be  a cut-set.
 
 Todorcevic \cite{TodOCA} has shown that, since
 there is no uncountable $R$-homogeneous set,
 the family $F\subset \mathbb N^{\mathbb N}$ must be covered by a
 countable family $\{F_n : n\in\omega\}$ of sets each with the property
 that $\left[F_n\right]^2$ is disjoint from $R$. 
 Using our separable metrizable topology on $F$
 and the fact that $R$ is open, this is equivalent
 to there being a family $\{ F_n  : n\in\omega\}$ of countable subsets
 of $F$ satisfying that for every $n$, the set of pairs from the
 closure of $F_n$  
 and that the union of the closures of the $F_n$'s cover $F$.
 If we prove that $\{p_\xi : \xi \in\omega_1\}$  forces
 there is no such sequence $\{F_n : n\in\omega\}$ then we
 have completed the proof of the theorem.  
 We use Lemma \ref{oracle} to do so.
 
 \medskip
 
  Let $\eta\in\omega_1$,
  $\{ p_\xi : \xi \in \eta\}\subset \Poset$, and the countable
  elementary submodel $\mathfrak U$ be as in  Lemma
  \ref{oracle}.
  We can assume further $\mathfrak U$ is equal to
   $M\cap H(\aleph_2)$ for some countable elementary 
   submodel $M\prec H(\aleph_3)$ satisfying
  that each of the objects $\{p_\xi:\xi\in\eta\}$, 
   $\{I_\alpha : \alpha\in\omega_1\}$, 
   $\{\dot {\mathcal C}_f : f\in 
   \mathbb N^{\mathbb N}\}$ are elements of $M$.
   Let us note that $R$ itself is an element of $V$ and
   that we may assume there is some $\xi<\eta$ such
   that $p_\xi$ forces that $R$ is the open set resulting
   from the conditions (1) and (2) above applied to the family
  $[F]^2$.
    It suffices
   to prove that if  
  $\{ \dot F_n : n\in\omega\}\in \mathfrak U$ 
  is a set of   $\Poset(\{p_\xi : \xi\in\eta\})$-names
  of countable subsets of $F$,
  then there is a choice of $p_\eta$ (as in Lemma \ref{oracle})
 that forces this sequence fails to have the covering properties
 mentioned in the previous paragraph.

 To do so, we
 first consider the forcing extension
 by the countable poset
 $\Poset(\{p_\xi : \xi\in \eta\})$.  
 Since the $\dot F_n$'s  are countable names of subsets of $F$
  we can fix a
  $\delta \in \mathfrak U\cap \omega_1$  
  such that it is forced that every element of
   $\bigcup\{ \dot F_n : n\in\omega\}$
   has as its domain an element of $\{ I_\beta :\beta <\delta\}$. 
 There is a maximal
 antichain $E_0$ of $\Poset(\{p_\xi : \xi\in \eta\})$ (in 
 $\mathfrak U$) where each $p\in E_0$ forces one of
 the following statements
 \begin{enumerate}
 \item there is an $f\in F\cap \mathfrak U$ such
 that $f$ is not in the closure of $\dot F_n$ for each $n\in\omega$,
 \item there is an $n\in\omega$  and a pair $\{f,f'\}\in \dot F_n\cap R$,
 \item each of the previous two statements (1) and (2) fail to hold.
 \end{enumerate} 
  Now consider any condition $q\in \Poset(\{p_\xi : \xi\in\eta\})$
   that is an extension of an element of $E_0$ in which
   condition (3) holds.  
   
    Working in $M$ we can ask if, for each $f\in \mathbb N^{\mathbb N}$,
    there is a condition $p_f <q$ in $\Poset(\{p_\xi : \xi\in\eta\})$
    and an integer $n_f$ satisfying that $f\restriction I_\delta$ is in the closure
    of $\dot F_{n_f}$.  We show that the answer to this question is ``no''.
    If such a pair, $p_f,n_f$, exists for all $f\in \mathbb N^{\mathbb N}$, 
    then there would be a single condition $\bar p$ and integer $\bar n$
    satisfying that $\tilde F = \{ f\in F: p_f = \bar p, n_f=\bar n\}$ is
    $<^*$-cofinal
    in $\mathbb N^{\mathbb N}$.  Since $\bar p$ forces
    that $\tilde F\restriction I_\delta$
     is a subset of the closure of $\dot F_{\bar n}$, 
     it follows from the failure of (2)
      that $[\tilde F\restriction I_\delta]^2$ is disjoint from $R$. 
       But this means that, for all $m\in I_\delta\setminus \bar m$,
        the family $\{ [c^m_f,d^m_f ] : f\in \tilde F\}$ 
        is linked. Since a linked family of compact intervals has non-empty
        intersection, this, together with the cofinality of $F$,
 would imply that $I_\delta$ is in the trivial ideal.
 
 So we have now proven there is some $f\in \mathbb N^{\mathbb N}$
 such that there is no pair $p_f, n_f$ as above. By elementarity
 there is such an $f\in M\cap {\mathbb N}^{\mathbb N}\subset
 \mathfrak U$. The failure of there being a pair $p_f, n_f$ is equivalent
 to the assertion, that for each $n\in\omega$, there is a
 subset $E(q,f,n)$ of $\Poset (\{p_\xi : \xi <\eta\})$ that is 
 dense below $q$ and for each $e\in E(q,f,n)$, there is a
 basic open neighborhood of $f\restriction I_\delta$ 
 in the topology on $S^{\mathbb N}$ such that
 $e$ forces that $\dot F_n$ is disjoint from this open set.
 
 It follows then the choice of $p_\eta$ as in Lemma \ref{oracle}
 will result in the family $\{ \dot F_n : n\in \omega\}$ considered
 as $\Poset (\mathcal F)$-names will not be  the dense sets
 for a countable   cover of $F$ consisting of sets
 whose set of pairs are disjoint from $R$.  The proof
is completed, i.e. a contradiction reached,  by applying
the PFA methodology discussed following Lemma \ref{getF}.  
 \end{proof}

\end{document}